\documentclass[11pt,a4paper]{article}
\usepackage[top=2.5cm, bottom=2.5cm, left=2.5cm, right=2.5cm]{geometry}
\usepackage[latin1]{inputenc}
\usepackage{csquotes}
\usepackage{amsmath}
\usepackage{amsthm}
\usepackage{amsfonts}
\usepackage{amssymb}
\usepackage{graphics}
\usepackage{float}
\usepackage[hang,flushmargin]{footmisc} 

\usepackage{amsmath,amsthm,amssymb,graphicx, multicol, array}
\usepackage{enumerate}
\usepackage{enumitem}
\usepackage{xcolor}

\usepackage[pagebackref,bookmarks,colorlinks,breaklinks]{hyperref}
\hypersetup{linkcolor=blue,citecolor=red,filecolor=blue,urlcolor=blue} 

\usepackage{tabto}

\numberwithin{equation}{section}

\usepackage{tikz}
\usepackage{pgfplots}

\usetikzlibrary{matrix}

\usepackage{mathrsfs}

\newtheorem{thm}{Theorem}[section]
\newtheorem{lem}{Lemma}[section]

\newtheorem{rmk}{Remark}[section]

\newcommand{\Z}{\mathbb{Z}}

\newcommand{\F}{\mathbb{F}}

\usepackage{fancyhdr}
\title{Largest Prime Factors of Quadratic Polynomials}
\date{}
\author{N. A. Carella}
\begin{document}

\thispagestyle{empty}
\date{}
\maketitle

\vskip .25 in 
\begin{abstract}
Let $x>1$ be a large number. This note shows that the largest prime factor of the quadratic product $\prod_{x\leq n\leq 2x}\left(n^2+1 \right)$ satisfies the relation $p \geq x^{3/2}$ as $x$ tends to infinity. This result improves the current estimate $p \geq x^{1.279}$. \let\thefootnote\relax\footnote{ \today \date{} 
		}
\end{abstract}
\setcounter{tocdepth}{1}
\tableofcontents

\section{Introduction }\label{S0022-IN}\hypertarget{S0022-IN}
Let $x>1$ be a large real number and let the symbol $P(n)\geq2$ denotes the largest prime divisor of an integer $n\geq2$. This note is concerned with the largest prime factor of the polynomial $n^2+1$. Apparently, there are several methods used to study the largest prime factors of polynomials $f(t)\in\Z[t]$ over the integers $\Z$. A technique, based on transcendental methods and similar concepts, attempts to estimate a lower bound of the largest prime divisor of $n^2+1$ for a single integer $n\geq1$, see \cite{MK1933}, \cite{CS1935}. Another technique, based on the prime number theorem, sieve methods and other advanced techniques, attempts to estimate a lower bound of the largest prime divisor of $n^2+1$ over a short interval, exempli gratia $n\in[x,2x]$, see \cite{NT1921}, \cite{HC1967}, \cite{MJ2023}. The latest result derived from transcendental methods claims that $P(n^2+1)\gg (\log\log n)^2/\log\log\log n$, this is proved in {\color{red}\cite[Theorem 1.1]{PH2024}}. The latest result derived from sieve methods claims that $P(n^2+1)\gg n^{1.279}$ and the conditional estimate $P(n^2+1)\gg n^{1.312}$, this is proved in {\color{red}\cite[Theorem 2]{MJ2023}}. An abridged history of the record values for the largest prime divisor $P(n^2+1)\geq2$ of the polynomial $f(n)=n^2+1$, expressed in terms of $n$, are tabulated in \autoref{table100-2}. \\

The largest prime factors of the more general products $\prod_{n\leq x}\left(n^2+a \right)$ associated with the irreducible polynomials $t^2+a\in \Z[t]$ are investigated in \cite{HC1967}, \cite{DI1982}, \cite{MJ2023} et alii. \\

This note presents a new result derived from the prime number theorem and elementary techniques. \\

\begin{thm} \label{thm0022.100-2}\hypertarget{thm0022.100-2} If $x$ is a large number then the finite product $\prod_{x\leq n\leq 2x}\left(n^2+1 \right)$ is divisible by a large prime $ p\in [x^{3/2},4x^2+1]  $.
\end{thm}	
\vskip .15 in 
This improves both the current unconditional estimate $p \geq x^{1.279}$ and the conditional estimate $p \geq x^{1.312}$ in {\color{red}\cite[Theorem 2]{MJ2023}} and {\color{red}\cite[Theorem 1.1]{PH2024}}. \\

A prime $p\mid n^2+1$ is \textit{primitive prime divisor} if $p\nmid m^2+1$ for all $m<n$. According to the results in \cite{EH2008} and \cite{HG2024} the largest primitive prime divisor of $n^2+1$ coincides with the largest prime divisor, and it  satisfies the same estimates given on \autoref{table100-2}.\\

The foundation materials are presented in \hyperlink{S0022-I}{Section} \ref{S0022-I}, \hyperlink{S0022-P}{Section} \ref{S0022-P} and \hyperlink{S0022-S}{Section} \ref{S0022-S}. Lastly, the proof of \hyperlink{thm0022.100-2}{Theorem} \ref{thm0022.100-2} appears in \hyperlink{S0044-T}{Section} \ref{S0044-T}. 
\vskip .1 in 
\begin{table}[h]
	\setlength{\tabcolsep}{0.5cm}
	\renewcommand{\arraystretch}{1.750092}
	\setlength{\arrayrulewidth}{0.75pt}
	\centering
	\begin{tabular}{l|l}
		
		Estimate of the largest prime divisor	& Author and reference \\
		\hline
		$P(n^2+1)/n\to \infty $ as $ n\to\infty$	&  Chebyshev, see \cite{HC1967}.\\
		\hline
		$P(n^2+1)\geq n(\log n)^{\varepsilon}$ as $ n\to\infty$	& Nagell, see \cite{NT1921}. \\
		\hline
		$P(n^2+1)\geq n(\log n)^{c\log\log\log n}$ as $ n\to\infty$		&  Erdos, see \cite{EP1952}.\\
		\hline
		$P(n^2+1)\geq n^{11/10}$ as $ n\to\infty$		&  Hooley, see \cite{HC1967}.\\
		\hline
		$P(n^2+1)\geq n^{1.202}$ as $ n\to\infty$		&  Iwaniec, see \cite{DI1982}.\\
		\hline
		$P(n^2+1)\geq n^{1.279}$ as $ n\to\infty$		&  Merikoski, see \cite{MJ2023}.\\
	\end{tabular}  
	\vskip .1 in		
	\caption{Estimated magnitude of prime divisor of $n^2+1$}
	\label{table100-2}
\end{table}

\section{Integer Solutions of Polynomials in Short Intervals}\label{S0022-I}\hypertarget{S0022-I}
Standard techniques for counting the number of integers solutions of congruence equations modulo a prime are employed here. This counting method is quite similar to the analysis in {\color{red}\cite[Equation (3)]{HC1967}}. The symbol $\{z\}=z-[z]$ denotes the fractional part function.\\

\begin{lem} \label{lem0022.200-2}\hypertarget{lem0022.200-2} Let $x\geq x_0$ be a large number and let $p\in [x,2x]$ be a prime such that $p\equiv 1 \bmod 4$, then
	\begin{equation}\label{eq0022.100-2a}
		\sum_{\substack{x\leq n\leq 2x\\n^2+1\equiv 0 \bmod p}}1\leq\frac{2x}{p}+\left\{\frac{x- i}{p}\right\}+\left\{\frac{x+ i}{p}\right\},
	\end{equation}
	where $| i|<p/2$ are the roots of the congruence $n^2+1\equiv 0\bmod p$.
\end{lem}

\begin{proof}[\textbf{Proof}]  The finite field $\F_p$ of characteristic $p\equiv 1 \bmod 4$  contains the complex $4$th roots of unity $\omega=e^{i2\pi /4}=i$. Moreover, by Lagrange theorem for polynomials over finite fields, the congruence equation $n^2+1\equiv 0\bmod p$ has $2$ solutions $\{-i,\,i\}$. Now, the equivalent class of the root $i=-b$ has 
	\begin{equation}\label{eq0022.200-2d}
		\left[\frac{2x+b}{p}\right]-	\left[\frac{x+b}{p}\right]=\frac{x}{p}-\left\{\frac{2x+ b}{p}\right\}+\left\{\frac{x+ b}{p}\right\}
	\end{equation}
	integer solutions in the range $x\leq n=ap- b\leq 2x$ with $a\geq1$ and $|b|<p/2$. Summing the contributions in \eqref{eq0022.200-2d} for both roots gives the total number
	\begin{eqnarray}\label{eq0022.200-2e}
		N_p(x,2x)&=&\left[\frac{2x- b}{p}\right]-	\left[\frac{x- b}{p}\right]+\left[\frac{2x+ b}{p}\right]-	\left[\frac{x+ b}{p}\right]\\
		&=&\frac{x}{p}-\left\{\frac{2x+ b}{p}\right\}+\left\{\frac{x+ b}{p}\right\}+
		\frac{x}{p}-\left\{\frac{2x- b}{p}\right\}+\left\{\frac{x- b}{p}\right\}\nonumber\\
		&\leq& \frac{2x}{p}+\left\{\frac{x+ b}{p}\right\}+\left\{\frac{x-b}{p}\right\}\nonumber
	\end{eqnarray}
	of integer solutions in the range $x\leq n\leq 2x$. 
\end{proof}

\begin{rmk} \label{rmk0022.200-2k}{\normalfont Evidently, the asymptotic approximation 
		\begin{equation}\label{eq0022.100-2k}
			\sum_{\substack{x\leq n\leq 2x\\n^2+1\equiv 0 \bmod p}}1=\frac{2x}{p}+O(1)
		\end{equation}
		is simpler and more manageable, but it does not work, it leads to a trivial result. However, both the exact expression and the approximation in \eqref{eq0022.200-2e} lead to nontrivial results. }
\end{rmk}

\section{Primary Term for Quadratic Polynomials}\label{S0022-P}\hypertarget{S0022-P}
The primary term $R(x)$ appearing in the proof of the main result, see \eqref{eq0044.400-2m}, is computed in this Section.\\

The counting measure for the primes $p\equiv a \bmod q$ over the short interval $[1,z]$ is approximated by the asymptotic formula
\begin{equation}\label{eq0022.300a}
	\pi ( z,q,a)=\frac{1}{\varphi(q)}\frac{z}{\log z}+O\left( \frac{z}{(\log z)^2}\right),
\end{equation}
where $\gcd(a,q)=1$, confer {\color{red}\cite[p.\ 424]{IK2004}}, {\color{red}\cite[p.\ 385]{MV2007}}.
\begin{lem} \label{lem0022.300-E}\hypertarget{lem0022.300-E} If $x\geq1$ is a large number and $a<q\leq (\log x)^C$ is a pair of relatively prime integers, where $c\geq0$ is a constant, then
	$$\sum_{\substack{p\leq x\\p\equiv a \bmod q}}\frac{\log p}{p}=\frac{1}{\varphi(q)}\log x +O\left(\log \log x\right),$$
where $\varphi(q)$ is the totient function.	
\end{lem}
\begin{proof}[\textbf{Proof}]  Assume the Siegel-Walfisz theorem. Evaluate the integral representation to derived the asymptotic formula
	\begin{eqnarray}\label{eq0022.300-E2}
		\sum_{\substack{p\leq x\\p\equiv a \bmod q}}\frac{\log p}{p}
		&=&\int_2^{x}\frac{\log t}{t}\,d\pi ( t,q,a)\\[.3cm]
		&=&\frac{\log x}{x} \cdot \left(\frac{1}{\varphi(q)} \frac{x}{\log x} +O\left( \frac{x}{(\log x)^2}\right)\right)\nonumber\\[.4cm]
		&&\hskip .4 in - \int_2^x\left( -\frac{\log t}{t^2}+\frac{1}{t^2}\right) \left( \frac{1}{\varphi(q)}\frac{t}{\log t}+O\left( \frac{t}{(\log t)^2}\right)\right)  dt\nonumber\\[.4cm]
		&=&\frac{1}{\varphi(q)}\int_2^x\left( \frac{1}{t}+O\left( \frac{1}{t\log t}\right)\right)dt+O(x)\nonumber\\[.3cm]
		&=&\frac{1}{\varphi(q)}\log x +O\left(\log \log x\right)\nonumber.
	\end{eqnarray}
\end{proof}

\begin{lem} \label{lem0022.300-2}\hypertarget{lem0022.300-2} Fix a parameter $\delta \geq 0$. If $x\geq1$ is a large number, then
	$$2x\sum_{\substack{p\leq x^{1+\delta}\\p\equiv 1 \bmod 4}}\frac{\log p}{p}=(1+\delta)x\log x +O\left(x\log \log x\right)$$
\end{lem}
\begin{proof}[\textbf{Proof}]  The evaluation of the asymptotic formula in \hyperlink{lem0022.300-E}{Lemma} \ref{lem0022.300-E} at $q=4$, $a=1$ and $x^{1+\delta}$ completes the verification.
\end{proof}

\section{Secondary Term for Quadratic Polynomials}\label{S0022-S}\hypertarget{S0022-S}
The secondary term $S(x)$ appearing in the proof of the main result, see \eqref{eq0044.400-2m}, is computed in this Section.\\

\begin{lem} \label{lem0022.400-2}\hypertarget{lem0022.400-2}  Fix a parameter $\delta \geq 0$. If $x\geq1$ is a large number, then
	$$ S(x)=\sum_{\substack{p\leq x^{1+\delta}\\p\equiv 1 \bmod 4}}
	\left\{\frac{x\pm b}{p}\right\}\log p	\leq \delta x\log x +O(x\log \log x).$$
\end{lem}
\begin{proof}[\textbf{Proof}]  For any prime $p>x\pm b$ the fractional part function reduces to the identity function 
	\begin{equation}\label{eq0022.400-2b}
		\left\{\frac{x\pm b}{p}\right\}\log p=\frac{(x\pm b)\log p}{p}.
	\end{equation}
	In light of this observation consider the dyadic partition
	\begin{equation}\label{eq0022.400-2c}
		\sum_{\substack{p\leq x^{1+\delta}\\p\equiv 1 \bmod 4}}\left\{\frac{x\pm b}{p}\right\}\log p	=\sum_{\substack{p\leq x\pm b\\p\equiv 1 \bmod 4}}\left\{\frac{x\pm b}{p}\right\}\log p+\sum_{\substack{x\pm b< p\leq x^{1+\delta}\\p\equiv 1 \bmod 4}}\frac{(x\pm b)\log p}{p}.
	\end{equation} 
	The first subsum in \eqref{eq0022.400-2c} has the upper bound
	\begin{eqnarray}\label{eq0022.400-2d}
		\sum_{\substack{p\leq x\pm b\\p\equiv 1 \bmod 4}}\left\{\frac{x\pm b}{p}\right\}\log p\leq(\log x)\sum_{\substack{p\leq x\pm b\\p\equiv 1 \bmod 4}}1	=O(x),
	\end{eqnarray}
	and an upper bound of the second subsum in \eqref{eq0022.400-2c} is computed in \hyperlink{lem0022.450-2}{Lemma} \ref{lem0022.450-2}.
	Adding these estimates completes the proof.
\end{proof}

\begin{lem} \label{lem0022.450-2}\hypertarget{lem0022.450-2}  Fix a parameter $\delta \geq 0$. If $x\geq1$ is a large number, then
	$$ \sum_{\substack{x\pm b< p\leq x^{1+\delta}\\p\equiv 1 \bmod 4}}\frac{(x\pm b)\log p}{p}
	\leq \delta x\log x +O(x\log \log x).$$
\end{lem}
\begin{proof}[\textbf{Proof}] Expand the short formula into 4 finite sums:
	\begin{eqnarray}\label{eq0022.450-2c}
		\sum_{\substack{x\pm b< p\leq x^{1+\delta}\\p\equiv 1 \bmod 4}}\frac{(x\pm b)\log p}{p}&=&x\sum_{\substack{x\pm b< p\leq x^{1+\delta}\\p\equiv 1 \bmod 4}}\frac{\log p}{p}\pm	\sum_{\substack{x\pm b< p\leq x^{1+\delta}\\p\equiv 1 \bmod 4}}\frac{b\log p}{p}	\\[.4cm]
		&=&x\sum_{\substack{x-b< p\leq x^{1+\delta}\\p\equiv 1 \bmod 4}}\frac{\log p}{p}-	\sum_{\substack{x-b< p\leq x^{1+\delta}\\p\equiv 1 \bmod 4}}\frac{b\log p}{p} \nonumber\\[.4cm]
		&&\hskip 1 in +x\sum_{\substack{x+b< p\leq x^{1+\delta}\\p\equiv 1 \bmod 4}}\frac{\log p}{p}+	\sum_{\substack{x+b< p\leq x^{1+\delta}\\p\equiv 1 \bmod 4}}\frac{b\log p}{p}\nonumber.
	\end{eqnarray}
	Extending the indices on the third and the fourth finite sums in \eqref{eq0022.450-2c} from the maximal $x+b$ to the minimal $x-b$ produces a simpler inequality
	\begin{eqnarray}\label{eq0022.450-2d}
		\sum_{\substack{x\pm b< p\leq x^{1+\delta}\\p\equiv 1 \bmod 4}}\frac{(x\pm b)\log p}{p}
		&\leq&x\sum_{\substack{x-b< p\leq x^{1+\delta}\\p\equiv 1 \bmod 4}}\frac{\log p}{p}-	\sum_{\substack{x-b< p\leq x^{1+\delta}\\p\equiv 1 \bmod 4}}\frac{b\log p}{p} \\[.4cm]
		&& +x\sum_{\substack{x-b< p\leq x^{1+\delta}\\p\equiv 1 \bmod 4}}\frac{\log p}{p}+	\sum_{\substack{x-b< p\leq x^{1+\delta}\\p\equiv 1 \bmod 4}}\frac{b\log p}{p}\nonumber\\[.4cm]
		&\leq&2x\sum_{\substack{x-b< p\leq x^{1+\delta}\\p\equiv 1 \bmod 4}}\frac{\log p}{p}	\nonumber.
	\end{eqnarray}
Set $q=4$, $a=1$ and evaluate the asymptotic formula in \hyperlink{lem0022.300-E}{Lemma} \ref{lem0022.300-E} over the short interval $[x-b, x^{1+\delta}]$ and simplify it.  
	\begin{eqnarray}\label{eq0022.450-2e}
		S(x)&\leq&2x\sum_{\substack{x-b< p\leq x^{1+\delta}\\p\equiv 1 \bmod 4}}\frac{\log p}{p}\\[.4cm]
		&\leq&2x\left[ \frac{1}{2}\left(\log x^{1+\delta} +O(\log \log x^{1+\delta})\right)-\frac{1}{2}\left(\log (x- b) +O(\log \log (x- b))\right)\right]   \nonumber\\[.4cm]
		&\leq&\delta x\log x +O(x\log \log x)\nonumber.
	\end{eqnarray}
This completes the verification.
\end{proof}

\section{Result for Quadratic Polynomials}\label{S0044-T}\hypertarget{S0044-T}
The proof of \hyperlink{thm0022.100-2}{Theorem} \ref{thm0022.100-2} is the topic of this section. Essentially, the analysis is the same as the previous attempts up to the expression \eqref{eq0044.400-2d}, see {\color{red}\cite[Equation (2)]{HC1967}} and more recently in {\color{red}\cite[Equation (1.1)]{MJ2023}}. Very complicated methods has been developed to estimate the inner sum in \eqref{eq0044.400-2d}. Here a key observation regarding the estimation of the secondary term $S(x)$ in \eqref{eq0044.400-2m} was made in \hyperlink{lem0022.400-2}{Lemma} \ref{lem0022.400-2}. This reduces this analysis to an elementary problem.\\

\begin{proof}[\textbf{\color{blue}Proof of \hyperlink{thm0022.100-2}{Theorem} {\normalfont \ref{thm0022.100-2}}}]
	Observe that for any large number $x$, the finite product  
	\begin{equation}\label{eq0044.400-2a}
		\prod_{x \leq n\leq 2x}\left(n^2+1 \right) \geq[2x]!^2-[x+1]!^2
	\end{equation}
	has an effective lower bound. In terms of finite sum this is equivalent to
	\begin{equation}\label{eq0044.400-2b}
		\sum_{x \leq n\leq 2x}\log(n^2+1 ) \geq 2x\log x+O(x).
	\end{equation}
	Let $\Lambda(n)$ be the vonMangoldt function. Employing the identity $\log n=\sum_{d\mid n}\Lambda(d)$ to rewrite the inequality \eqref{eq0044.400-2b} yields 
	\begin{eqnarray}\label{eq0044.400-2c}
		2x\log x+O(x) &\leq&\sum_{x \leq n\leq 2x}\log\left(n^2+1 \right)\\
		&=&\sum_{x \leq n\leq 2x,}\sum_{d\mid n^2+1}\Lambda(d)\nonumber \\
		&=&	\sum_{d\leq 4x^2+1}\Lambda(d)\sum_{\substack{x \leq n\leq 2x\\n^2+1\equiv 0 \bmod d}}1\nonumber.
	\end{eqnarray}

	Since $t^2+1\in\Z[t]$ is irreducible over the integers, the last expression implies the inequality
	\begin{equation}\label{eq0044.400-2d}
		2x\log x+O(x)\leq \sum_{p\leq 4x^2+1}\log p\sum_{\substack{x \leq n\leq 2x\\n^2+1\equiv 0 \bmod p}}1,
	\end{equation}
	where the inner sum is indexed by the prime divisors $p\leq 4x^2+1$ of $n^2+1$ for $x \leq n\leq 2x$.\\
	
	As in \cite{HC1967} let $P_x=x^{1+\delta}\leq 4x^2+1$, where $\delta\ge0$.  Applying \hyperlink{lem0022.200-2}{Lemma} \ref{lem0022.200-2} to the inequality \eqref{eq0044.400-2d} returns
	
	\begin{eqnarray}\label{eq0044.400-2m}
		N(x)&=&\sum_{p\leq x^{1+\delta}}	\log p\sum_{\substack{x \leq n\leq 2x\\n^2+1\equiv 0 \bmod p}}1\\[.3cm]
		&\leq&\sum_{\substack{p\leq x^{1+\delta}\\p\equiv 1 \bmod 4}}\log p\left( \frac{2x}{p}+\left\{\frac{x- b}{p}\right\}+\left\{\frac{x+ b}{p}\right\}\right)\nonumber\\[.3cm]
		&=&2x\sum_{\substack{p\leq p\leq x^{1+\delta}\\p\equiv 1 \bmod 4}}\frac{\log p}{p}    +\sum_{\substack{p\leq x^{1+\delta}\\p\equiv 1 \bmod 4}}\left\{\frac{x- b}{p}\right\}\log p +\sum_{\substack{p\leq x^{1+\delta}\\p\equiv 1 \bmod 8}}\left\{\frac{x+ b}{p}\right\}\log p\nonumber\\[.3cm]
		&=&R(x)+\,S(x)\nonumber.
	\end{eqnarray}
The primary term $R(x)$ is computed in \hyperlink{lem0022.300-2}{Lemma} \ref{lem0022.300-2} and the secondary term $S(x)$ is estimated in \hyperlink{lem0022.400-2}{Lemma} \ref{lem0022.400-2}. Summing these expressions returns
	\begin{eqnarray}\label{eq0044.400-2p}
		2x\log x+O(x)&\leq&\sum_{p\leq x^{1+\delta}}	\log p\sum_{\substack{x \leq n\leq 2x\\n^2+1\equiv 0 \bmod p}}1\\[.3cm]
		&=&	R(x)+\,S(x)\nonumber\\[.3cm]
		&\leq&(1+\delta)x\log x +O\left(x\log \log x\right)+ \delta x\log x +O(x\log \log x)\nonumber\\[.3cm]
		&\leq&(1+2\delta)x\log x +O\left(x\log \log x\right)\nonumber.
	\end{eqnarray}
	This implies that for any parameter $\delta<1/2$, the inequality \eqref{eq0044.400-2p}
	is a contradiction. Therefore, the parameter $\delta\geq 1/2$. In particular, the double finite sum 
	\begin{equation}\label{eq0044.400-2r}
		\sum_{x^{3/2}\leq p\leq 4x^2+1}\log p\sum_{\substack{x \leq n\leq 2x\\n^2+1\equiv 0 \bmod p}}1>0
	\end{equation}
	over the short interval $[x^{3/2},4x^2+1]$ does not vanish. Therefore, the product 
	\begin{equation}\label{eq0044.400-2s}
		\prod_{x \leq n\leq 2x}\left(n^2+1 \right) 
	\end{equation}
	has a large prime divisor $p\in [x^{3/2},4x^2+1]$.
\end{proof}


\end{document}